\documentclass{article}

\usepackage{amsmath,amscd,amsfonts,amssymb,amsthm}
\usepackage{subfig}
\usepackage{graphicx}
\usepackage{epstopdf}
\usepackage{a4wide}
\usepackage{caption}

\usepackage{graphicx}
\usepackage{multirow}
\usepackage{hyperref}

\theoremstyle{definition}

\newtheorem{theorem}{Theorem}

\newtheorem{lemma}{Lemma}
\theoremstyle{remark}
\newtheorem{remark}{Remark}

\begin{document}

\title{Modelling some real phenomena by fractional differential equations}

\author{Ricardo Almeida$^1$, Nuno R. O. Bastos$^{1,2}$ and M. Teresa T. Monteiro$^3$}

\date{$^1$Center for Research and Development in Mathematics and Applications (CIDMA)\\
Department of Mathematics, University of Aveiro, 3810--193 Aveiro, Portugal\\
$^2$Centre for the Study of Education, Technologies and Health (CI$\&$DETS)\\
Department of Mathematics, School of Technology and Management of Viseu\\Polytechnic Institute of Viseu, 3504--510 Viseu, Portugal\\
$^3$Algoritmi R$\&$D Center, Department of Production and Systems\\
University of Minho, Campus de Gualtar, 4710–--057 Braga, Portugal}

\maketitle

\begin{abstract}

This paper deals with fractional differential equations, with dependence on  a Caputo fractional derivative of real order. The goal is to show, based on concrete examples and experimental data from several experiments, that fractional differential equations may model more efficiently certain problems than ordinary differential equations. A numerical optimization approach based on least squares approximation is used to determine the order of the fractional operator that better describes real data, as well as other related parameters.
\end{abstract}

\textbf{keywords:} {Fractional calculus, fractional differential equation, numerical optimization}

\section{Introduction}

\subsection{Fractional calculus}

We start with a review on fractional calculus, as presented in e.g. \cite{Kilbas,Podlubny,Samko}. Fractional calculus is an extension of ordinary calculus, in a way that derivatives and integrals are defined for arbitrary real order. In some phenomena, fractional operators allow to model better than ordinary derivatives and ordinary integrals, and can represent more efficiently systems with high-order dynamics and complex nonlinear phenomena. This is due to two main reasons; first, we have the freedom to choose any order for the derivative and integral operators, and not be restricted to integer-order only. Secondly, fractional order derivatives depend not only on local conditions but also on the past, useful when the system has a long-term memory.

In a famous letter dated 1695, L'Hopital's asked Leibniz  what would be the derivative of order $\alpha=1/2$, and the response of Leibniz ``An apparent paradox, from which one day useful consequences will be drawn" became the birth of fractional calculus. Along the centuries, several attempts were made to define fractional operators. For example, in 1730, using the formula
$$\frac{d^n x^m}{dx^n}= m(m-1)\ldots (m-n+1)x^{m-n}=\frac{\Gamma(m+1)}{\Gamma(m-n+1)}x^{m-n},$$
Euler obtained the following expression
$$\frac{d^{1/2} x}{dx^{1/2}}=\sqrt{\frac{4x}{\pi}}.$$
Here, $\Gamma$ denotes the Gamma function,
$$\Gamma(t)=\int_0^\infty s^{t-1}\mbox{e}^{-s}\,ds, \quad t\in\mathbb R \setminus \mathbb Z^-_0,$$
which is  an extension of the factorial function to real numbers. Two of the basic properties of the Gamma function are
$$\Gamma(m)=(m-1)!, \quad \mbox{for all }m\in\mathbb N,$$
and
$$\Gamma(t+1)=t \, \Gamma(t),\quad \mbox{for all } t\in\mathbb R \setminus \mathbb Z^-_0.$$
However, the most famous and important definition is due to Riemann. Starting with  Cauchy's formula
$$\int_a^t ds_1\,\int_a^{s_1}ds_2\,\ldots \int_a^{s_{n-1}}y(s_n) \, ds_n=\frac{1}{(n-1)!}\int_a^t (t-s)^{n-1}y(s) ds,$$
Riemann defined a fractional integral type of order $\alpha>0$ of a function $y:[a,b]\to\mathbb R$ as
$${_aI_t^\alpha}y(t)=\frac{1}{\Gamma(\alpha)}\int_a^t (t-s)^{\alpha-1}y(s) ds.$$
Fractional derivatives are defined using the fractional integral idea. The Riemann--Liouville fractional derivative of order $\alpha>0$ is given by
$${_aD_t^\alpha}y(t)=\left(\frac{d}{dt}\right)^n{_aI_t^{n-\alpha}}y(t)=\frac{1}{\Gamma(n-\alpha)}\left(\frac{d}{dt}\right)^n\int_a^t (t-s)^{n-\alpha-1}y(s) ds,$$
where $n\in\mathbb N$ is such that $\alpha\in(n-1,n)$. These definitions are probably the most important with respect to  fractional operators, and until the $20th$ century the subject was relevant in pure mathematics only. Nowadays, this is an important field not only in mathematics but also in other sciences, engineering, economics, etc. In fact, using these more general concepts, we can describe better certain real world phenomena which are not possible using integer-order derivatives. For example, we can find applications of fractional calculus in mechanics \cite{Atanackovic}, engineering \cite{Machado}, viscoelasticity \cite{Meral},  dynamical systems \cite{Tarasov}, etc.
Although the Riemann--Liouville fractional derivative is of great importance, from historial reasons, it may not be suitable to model some real world phenomena. So, different operators are considered in the literature, for example, the Caputo fractional derivative. It has proven its applicability due to two reasons: the fractional derivative of a constant is zero and the initial value problems depend on integer-order derivatives only.
The Caputo fractional derivative of a function $y:[a,b]\to\mathbb R$ of order $\alpha>0$ is defined by
$${_a^CD_t^{\alpha}}y(t)={_aI_t^{n-\alpha}}y^{(n)}(t)=\frac{1}{\Gamma(n-\alpha)}\int_a^t (t-s)^{n-\alpha-1}y^{(n)}(s) ds,$$
where $n=[\alpha]+1$. In particular, when $\alpha\in(0,1)$, we obtain
$${_a^CD_t^{\alpha}}y(t)=\frac{1}{\Gamma(1-\alpha)}\int_a^t (t-s)^{-\alpha}y'(s)ds.$$
The fractional integral and fractional derivative are inverse operations, in the sense that (see \cite{Kilbas}):

\begin{lemma}\label{2.22} Let $\alpha>0$ and $n\in\mathbb N$ be such that $\alpha\in(n-1,n)$. If $y\in AC^n[a,b]$, or $y\in C^n[a,b]$, then
$${_aI_t^\alpha} \, {_a^CD_t^{\alpha}}y(t)=y(t)-\sum_{k=0}^{n-1}\frac{y^{(k)}(a)}{k!}(t-a)^k.$$
\end{lemma}

\noindent We can obtain ordinary derivatives by computing the limit when $\alpha\to n\in\mathbb N$. In fact, we have the following property:
$$\lim_{\alpha\to n^-} {_a^CD_t^{\alpha}}y(t) =y^{(n)}(t) \quad \mbox{and} \quad \lim_{\alpha\to n^+}{_a^CD_t^{\alpha}}y(t)=y^{(n)}(t)-y^{(n)}(a).$$
For example, let $y^*(t)=t$, with $t\geq0$. Then, for $\alpha\in(0,1)$, we have
$${_0^CD_t^{\alpha}}y^*(t)=\frac{1}{\Gamma(2-\alpha)}t^{1-\alpha},$$
while for $\alpha>1$, we have ${_0^CD_t^{\alpha}}y^*(t)=0$ (see Fig. \ref{fig:t}).

\begin{figure}[h]
  \centering
  \includegraphics[width=0.55\linewidth]{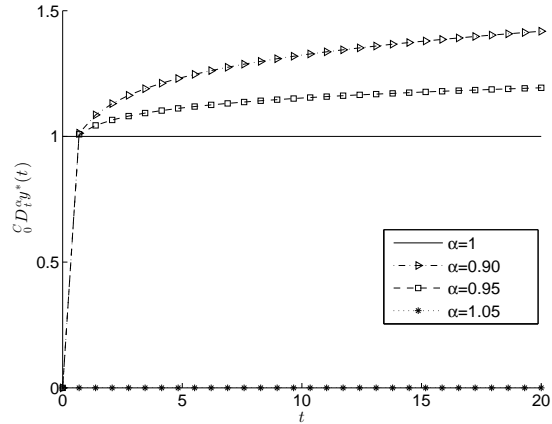}
  \caption{Derivative of $y^*$ and its fractional derivatives.}\label{fig:t}
\end{figure}

However, we have the following result.

\begin{theorem}\label{teo1} Let $f:[a,b]\times\mathbb R\to\mathbb R$ be a continuous function and $\alpha>0$ a real.
If $y^*_\alpha$ is a solution of the system
$$\left\{
\begin{array}{l}
{_a^CD_t^{\alpha}}y(t)=f(t,y), \quad \alpha\in(0,1)\\
y(a)=y_a,
\end{array}
\right.$$
and if the limit
$$\lim_{\alpha\to1^\pm}y^*_\alpha(t):= y^*(t)$$
exists for all $t\in[a,b]$, then $y^*$ is solution of the Cauchy problem
$$\left\{
\begin{array}{l}
y'(t)=f(t,y)\\
y(a)=y_a.
\end{array}
\right.$$
\end{theorem}

\begin{proof}
Starting with equation ${_a^CD_t^{\alpha}}y^*_\alpha(t)=f(t,y^*_\alpha)$, applying the $I^\alpha$ operator to both sides and using Lemma \ref{2.22}, we obtain that $y^*_\alpha$ satisfies the Volterra equation
$$y^*_\alpha(t)=y_a+\frac{1}{\Gamma(\alpha)}\int_a^t(t-s)^{\alpha-1}f(s,y^*_\alpha(s))\,ds.$$
Taking the limit $\alpha\to1^\pm$, we get
$$y^*(t)=y_a+\int_a^tf(s,y^*(s))\,ds,$$
that is, $y^*$ satisfies the Cauchy problem.
\end{proof}

\begin{remark} The idea of Theorem \ref{teo1} is the following. We first consider a fractional differential equation of order $\alpha\in(0,1)$, and as $\alpha\to1^-$, the fractional system converts into an ordinary Cauchy problem. Then, in order to obtain a higher accuracy for the method, we consider the solution of the fractional system $y^\star$, depending on the parameter $\alpha>0$, and compute the limit $\alpha\to1^\pm$.
\end{remark}

\subsection{Computational issues}

The MATLAB numerical experiments  to find the model parameters as well as the non-integer derivatives that fits better in the model is carried out.
The numerical experiments were done in MATLAB \cite{Matlab} using the routine \texttt{lsqcurvefit} that solves nonlinear data-fitting problems in least-squares sense, that is, given data $(y_i)_{i=1\ldots m}$, and a model $(M_i)_{i=1\ldots m}$, the error is:
$$E=\sum_{i=1}^m(y_i-M_i)^2.$$
This routine is based on an iterative method with local convergence, \emph{i.e.}, depending on the initial approximation to the parameters to estimate. The initial approximations were found, based on the data analysis of each model.  In the computational tests the trust-region-reflective algorithm was selected.
In all the studied cases, we will see that the fractional approach is more close to the data than the classical one. To test the efficiency of the model, we compare the error of the classical model $E_{classical}$ with the error given by the fractional model $E_{fractional}$:
  $$\left|\frac{E_{classical}-E_{fractional}}{E_{classical}}\right|.$$

\subsection{Outline}

This article is organized as follows. In each section we present a real problem, described by an ordinary or system of ordinary differential equations. We then consider the same problem, but modeled by a fractional or system of fractional differential equations, and compare which of the two models are more suitable to describe the process, based on real experimental data.
  Given an ordinary differential equation $y'(t)=f(t,y)$, we replace it by the fractional differential equation ${_a^CD_t^{\alpha}}y(t)=f(t,y)$, with $\alpha\in(0,1)$. When we consider the limit $\alpha\to1^-$, we obtain the initial one. If we consider $\alpha\in(1,2)$ as well, and then take the limit $\alpha\to1^+$, we would get $y'(t)-y'(0)=f(t,y)$, and this is the reason why we will neglect this case firstly. After computing the solution of the fractional differential equation, which depends on $\alpha$, we consider the function with fractional order $\alpha$ on the interval $(0,2)$ in order to increase the accuracy of the method. Three distinct problems are studied in this paper.
Section \ref{sec:popul} deals with the exponential law of growth, and we replace it by the Mittag--Leffler function, which is a generalization of the exponential function.
In Section \ref{sec:BAL} we deal with an application related to Blood Alcohol Level  using integer-order derivatives.
In the final Section \ref{sec:video} we study a model of a tape counter readings at an instant time.
In each case, we compute the values of the parameters that better fit with the given data, and also the error in each of both approaches.

\section{World Population Growth}\label{sec:popul}
\subsection{Standard approach}

There exist several attempts to describe the World Population Growth \cite{Smith}. The simplest model is the following, known as the \textit{Malthusian law of population growth}, which is used to predict populations under ideal conditions. Let $N(t)$ be the number of individuals in a population at time $t$ , $B$ and $M$ the birth and mortality rates, respectively, so that the net growth rate is given by
\begin{equation}\label{popul-Eq-Dif}N'(t)= (B-M)N(t)=PN(t),\end{equation}
 where $P:=B-M$ is the production rate.  Here, we assume that $B$ and $M$ are constant, and thus $P$ is also constant.
The solution of this differential equation is the function
\begin{equation}\label{popul}N(t)=N_0 \mbox{e}^{P t}, \quad t\geq0,\end{equation}
where $N_0$ is the population at $t=0$. Because of the solution \eqref{popul}, this model is also known as the exponential growth model.

\subsection{Fractional approach}

Considered now that the World Population Growth model is ruled by the fractional differential equation
\begin{equation}\label{popul-Eq-frac}{_0^CD_t^{\alpha}}N(t)=P N(t),\quad t \geq 0, \, \alpha \in(0,1).\end{equation}
Observe that, taking the limit $\alpha\to1^-$, Eq \eqref{popul-Eq-frac} converts into Eq. \ref{popul-Eq-Dif}, but if we consider $\alpha\in(1,2)$ and take the limit $\alpha\to1^+$, we obtain $N'(t)-N'(0)=PN(t)$.

Using Theorem 7.2 in \cite{Diethelm}, the solution of this fractional differential equation is the function
\begin{equation}\label{popul-frac} N(t)=N_0 E_\alpha(P t^\alpha),\end{equation}
where $E_\alpha$ is the  Mittag--Leffler function
$$E_\alpha(t)=\sum_{k=0}^\infty\frac{t^k}{\Gamma(\alpha k+1)}, \quad \, t\in\mathbb R.$$

Consider now function \eqref{popul-frac}, with $\alpha\in(0,2)$. Then, as $\alpha\to1^\pm$, we recover the solution for the classical problem \eqref{popul}.


\subsection{Numerical experiments}

For our numerical treatment of the problem, we find several databases with the world population through the centuries. Here, we use the one provided by the United Nations \cite{UN}, from year 1910 until 2010, consisting in 11 values, where the initial value is $N_0=1750$. For the classical approach, the production rate is
$$P\approx 1.3501\times 10^{-2}$$
and the error from the data with respect to the analytic solution \eqref{popul} is given by
$$E_{classical}\approx 7.0795\times 10^{5}.$$

If we take into consideration the fractional model \eqref{popul-frac}, for $\alpha\in(0,2)$ we obtain that the best values are
$$\alpha= 1.393298754843208    \quad \mbox{and} \quad P\approx 3.4399\times 10^{-3}$$
with error
$$E_{fractional}\approx 2.0506\times 10^{5}.$$

The gain of the efficiency in this procedures is
$$\frac{7.0795\times 10^{5}-2.0506\times 10^{5}}{7.0795\times 10^{5}}\approx 0.71.$$

The graph of Figure \ref{fig:FigPopulacaoM} illustrates the World Population Growth model with the data, the classical model and the fractional model.
\\
\begin{figure}[!htbp]
\centering
\hspace*{-15pt}\includegraphics[width=1\linewidth]{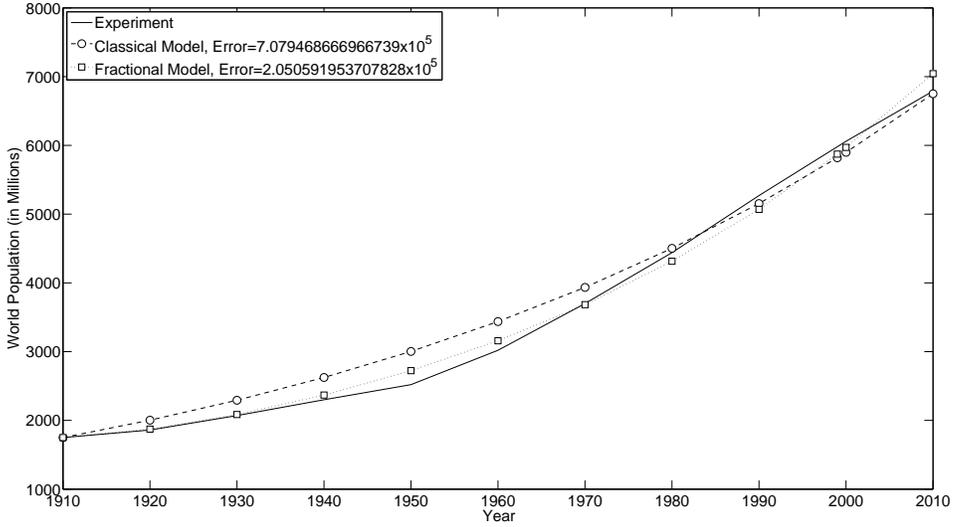}
\caption{World Population Growth model.}
\label{fig:FigPopulacaoM}
\end{figure}

\section{Blood Alcohol Level problem}\label{sec:BAL}
\subsection{Standard approach}
In \cite{Ludwin} we find a simple model to determine the level of Blood Alcohol, described by a system of two differential equations. Let $A$ represent the concentration of alcohol in the stomach and $B$ the concentration of alcohol in the blood. The problem is described by the following Cauchy system:
$$\left\{\begin{array}{l}
A'(t)=-k_1A(t)\\
B'(t)=k_1A(t)-k_2B(t)\\
A(0)=A_0\\
B(0)=0,
\end{array}\right.$$
where $A_0$ is the initial alcohol ingested by the subject and $k_1,k_2$ some real constants.
The solution of this system is given by the two functions
$$A(t)=A_0 \mbox{e}^{-k_1t}$$
and
$$B(t)=A_0\frac{k_1}{k_2-k_1}(\mbox{e}^{-k_1t}-\mbox{e}^{-k_2t}).$$
Also, in \cite{Ludwin} some experimental data is obtained in order to determine the arbitrary constants $k_1$ and $k_2$ (Table \ref{tab:BA1}). The time is in minutes and the Blood Alcohol Level (BAL) in mg/l.

\begin{center}
\begin{tabular}{|l|c|c|c|c|c|c|c|c|c|}
\hline
Time & 0 & 10& 20&30&45&80&90 &110& 170 \\
\hline
BAL&0& 150&200&160&130& 70&60&40&20\\
\hline
\end{tabular}\captionof{table}{Experimental data for Blood Alcohol Level.}\label{tab:BA1}
\end{center}

Using the data from the table, the values that minimize the Mean Absolute Error when the values from $B$ are fitted with the experimental data are:
$$A_0=245.8769, \quad k_1=0.109456, \quad k_2=0.017727,$$
and the error in this approximation is
$$E_{classical}=775.2225,$$
and for comparison, and get the following results in Table \ref{tab:BA2}

\begin{center}
\renewcommand{\tabcolsep}{3pt}
{\scriptsize
\begin{tabular}{|l|c|c|c|c|c|c|c|c|c|c|}
\hline
Time & 0 & 10& 20&30&45&80&90 &110& 170& Error \\
\hline
BAL&\multirow{2}{*}{0}& \multirow{2}{*}{150}&\multirow{2}{*}{200}&\multirow{2}{*}{160}&\multirow{2}{*}{130}&\multirow{2}{*}{70}&\multirow{2}{*}{60}&\multirow{2}{*}{40}
&\multirow{2}{*}{20}&\multirow{2}{*}{---}\\
(Experiment)  &&&&&&&&&&\\
\hline
BAL& \multirow{2}{*}{0.0000}&\multirow{2}{*}{147.5379}&\multirow{2}{*}{172.9499}&\multirow{2}{*}{161.3813}&\multirow{2}{*}{130.0021}& \multirow{2}{*}{71.0018}&\multirow{2}{*}{59.4910}&\multirow{2}{*}{41.7418}&\multirow{2}{*}{14.4100}&\multirow{2}{*}{775.2225}\\
(Classical)  &&&&&&&&&&\\
\hline
\end{tabular}
}\captionof{table}{Experimental data vs theoretical model.}\label{tab:BA2}
\end{center}

\subsection{Fractional approach}

Now we show that, if we consider the problem modeled by a system of fractional differential equations, we obtain a curve that better fit with the experimental results.

Let $\alpha, \beta \in (0,1)$ and consider the system of fractional differential equations
$$\left\{\begin{array}{l}
{_0^CD_t^{\alpha}}A(t)=-k_1A(t)\\
{_0^CD_t^{\beta}}B(t)=k_1A(t)-k_2B(t)\\
A(0)=A_0\\
B(0)=0.
\end{array}\right.$$

Using Theorem 7.2 in \cite{Diethelm}, the solution with respect to $A$ is
\begin{equation}\label{blood-frac}
A(t)=A_0E_\alpha(-k_1t^\alpha).\end{equation}
To determine $B$, it can be found as the solution of the fractional differential linear equation
$${_0^CD_t^{\beta}}B(t)=-k_2B(t)+k_1A_0E_\alpha(-k_1t^\alpha).$$
By Theorem 7.2 in \cite{Diethelm}, we obtain the solution with respect to $B$:
$$\begin{array}{ll}
B(t)&=\displaystyle\beta\int_0^t k_1A_0E_\alpha(-k_1(t-s)^\alpha)s^{\beta-1}E'_\beta(-k_2s^\beta)ds\\
&=\displaystyle\beta k_1A_0\int_0^t\sum_{m=0}^\infty\frac{(-k_1)^m(t-s)^{m\alpha}}{\Gamma(\alpha m+1)}s^{\beta-1} \sum_{n=0}^\infty\frac{(n+1)(-k_2)^ns^{n\beta}}{\Gamma(n\beta +\beta +1)}ds\\
&\displaystyle=\beta k_1A_0\sum_{m=0}^\infty\sum_{n=0}^\infty\frac{(-k_1)^m(-k_2)^n (n+1)}{\Gamma(\alpha m+1)\Gamma(n\beta+\beta+1)}\int_0^t(t-s)^{m\alpha}s^{n\beta+\beta-1}ds.
\end{array}$$
We remark that, as the expression inside the series is continuous and uniformly convergent, the previous calculations are valid.
Let us re-write function $B(\cdot)$.
Using the Beta function,
$$B(a,b)=\int_0^1t^{a-1}(1-t)^{b-1}dt, \quad a,b>0,$$
the following property
$$B(a,b)=\frac{\Gamma(a)\Gamma(b)}{\Gamma(a+b)},$$
and doing the change of variables $u=s/t$, we obtain
$$\int_0^t(t-s)^{m\alpha}s^{n\beta+\beta-1}ds=t^{m\alpha+n\beta+\beta}\int_0^1(1-u)^{m\alpha}u^{n\beta+\beta-1}du$$
$$=t^{m\alpha+n\beta+\beta}B(n\beta+\beta,m\alpha+1)=t^{m\alpha+n\beta+\beta}\frac{\Gamma(n\beta+\beta)\Gamma(m\alpha+1)}{\Gamma(n\beta+\beta+m\alpha+1)},$$
and so
\begin{equation}\label{alphabeta01}
B(t)=\displaystyle k_1A_0\sum_{m=0}^\infty\sum_{n=0}^\infty\frac{(-k_1)^m(-k_2)^n }{\Gamma(n\beta+\beta+m\alpha+1)}t^{m\alpha+n\beta+\beta}.
\end{equation}

\subsection{Numerical experiments}

For computational purposes we consider the upper bounds $n=m=45$ in Eq. \eqref{alphabeta01}, and we intend to determine the fractional orders $\alpha$ and $\beta$ on the interval $(0,2)$, the initial value $A_0$,  and the parameters $k_1$ and $k_2$ that best fit with the data. To that purpose, we obtain the values
$$A_0\approx373.0295, \quad k_1\approx 0.0643, \quad k_2\approx 0.0088,$$
with fractional orders
$$\alpha\approx1.1771 , \quad \beta\approx1.0052,$$
and the error in this approximation is
$$E_{fractional}\approx321.9677.$$

By this method, the efficiency has increased
$$\frac{775.2225-321.9677}{775.2225}\approx0.58.$$

In Table \ref{tab:BA3} we summarize the results, comparing the standard approach, using first-order derivatives, with ours using fractional-order derivatives.

\begin{center}
\renewcommand{\tabcolsep}{3pt}
{\scriptsize
\begin{tabular}{|l|c|c|c|c|c|c|c|c|c|c|}
\hline
Time & 0 & 10& 20&30&45&80&90 &110& 170& Error \\
\hline
BAL&\multirow{2}{*}{0}& \multirow{2}{*}{150}&\multirow{2}{*}{200}&\multirow{2}{*}{160}&\multirow{2}{*}{130}&\multirow{2}{*}{70}&\multirow{2}{*}{60}&\multirow{2}{*}{40}&\multirow{2}{*}{20}&
\multirow{2}{*}{---}\\
(Experiment)  &&&&&&&&&&\\
\hline
BAL& \multirow{2}{*}{0.0000}&\multirow{2}{*}{147.5379}&\multirow{2}{*}{172.9499}&\multirow{2}{*}{161.3813}&\multirow{2}{*}{130.0021}& \multirow{2}{*}{71.0018}&\multirow{2}{*}{59.4910}&\multirow{2}{*}{41.7418}&\multirow{2}{*}{14.4100}&\multirow{2}{*}{775.1745}\\
(Classical)  &&&&&&&&&&\\
\hline
BAL&\multirow{2}{*}{0.0000}&\multirow{2}{*}{155.7458}&\multirow{2}{*}{187.1950} &\multirow{2}{*}{169.6587} &\multirow{2}{*}{128.4871} &\multirow{2}{*}{69.3254}&\multirow{2}{*}{59.3669}&\multirow{2}{*}{43.7670} & \multirow{2}{*}{16.2108}&\multirow{2}{*}{321.9677}\\
(Fractional)  &&&&&&&&&&\\
\hline
\end{tabular}
}\captionof{table}{Results for the Blood Alcohol Level problem.}\label{tab:BA3}
\end{center}
The graph of Figure \ref{fig:figureBAL} illustrates the Blood Alcohol Level with experimental data, classical model and fractional model.
\\
\begin{figure}[!htbp]
\centering
\hspace*{-15pt}\includegraphics[width=1\linewidth]{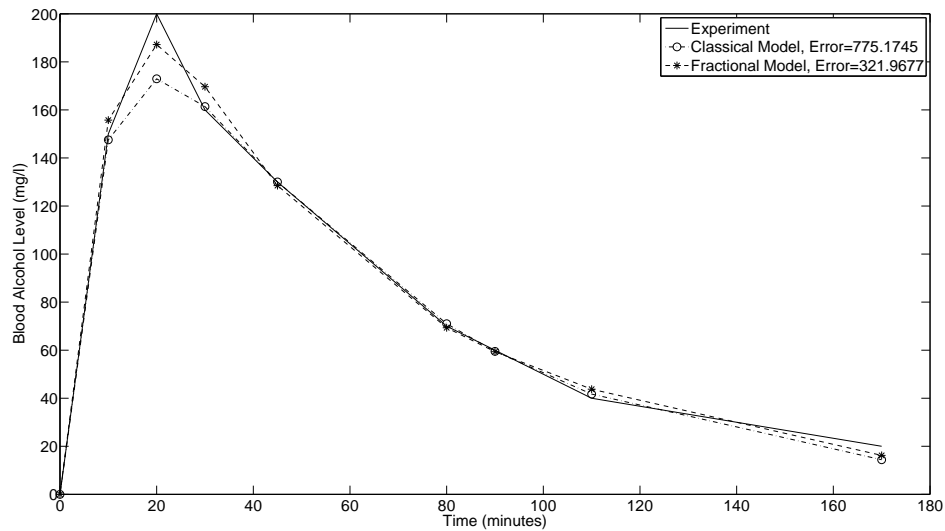}
\caption{The blood alcohol level problem.}
\label{fig:figureBAL}
\end{figure}

\section{Video Tape problem}\label{sec:video}
\subsection{Standard approach}
For the video tape problem, we intend to determine the tape counter readings $n(t)$ at an instant time $t>0$. For that, the problem is modeled in the following way. Let $c$ represents the thickness of the tape and $v$ the velocity of the tape, which is constant in time. Also, $R(t)$ indicates the radius of the wheel as is filled by the tape, and $A(t)$ measures the angle at moment $t$.


The number of revolutions is proportional to the angle, that is, $n(t)=kA(t)$, for some positive constant $k$. In \cite{Galphin}, an ordinary differential equation is obtained to describe the behavior of the model:
\begin{equation}\label{video:ODE}A'(t)=\frac{v}{R(0)\sqrt{bt+1}}, \quad \mbox{where} \, b=\frac{cv}{\pi R^2(0)}.\end{equation}
Using the initial condition $A(0)=0$, the solution is obtained:
$$A(t)=\frac{2v}{bR(0)}(\sqrt{bt+1}-1).$$
In conclusion, the tape counter readings is given by the expression
\begin{equation}\label{tape-standard}n(t)=a(\sqrt{bt+1}-1), \quad \mbox{where} \, a=\frac{2kv}{bR(0)}.\end{equation}

\subsection{Fractional approach}

In this section we consider the video tape problem modeled by a fractional differential equation of order $\alpha \in(0,1)$:
$$\left\{\begin{array}{l}
{_0^CD_t^{\alpha}}A(t)=\displaystyle\frac{v}{R(0)\sqrt{bt+1}}, \quad \mbox{where} \, b=\frac{cv}{\pi R^2(0)}\\
A(0)=0\\
\end{array}\right.$$
Applying the fractional integral operator ${_0I_t^{\alpha}}$ to both sides of the fractional differential equation, and using Lemma \ref{2.22}, that reads as
$${_0I_t^{\alpha}}{_0^CD_t^{\alpha}}A(t)=A(t)-A(0), \quad \mbox{if} \, \alpha \in(0,1), $$
we obtain that the solution is given by
\begin{equation}\label{tape-farc-1} n(t)=\frac{p}{\Gamma(\alpha)}\int_0^t\frac{(t-s)^{\alpha-1}}{\sqrt{bs+1}}ds, \quad \mbox{where} \, p=\frac{kv}{R(0)}, \, b=\frac{cv}{\pi R^2(0)}.\end{equation}

\subsection{Numerical experiments}

Using the experimental data available at \url{http://people.uncw.edu/lugo/MCP/DIFF_EQ/deproj/deproj.htm}, that counts the number of revolutions every $5$ minutes, from $t=0$ until $t=240$, we first determine the constants $a$ and $b$ for which the model given in Eq. \eqref{tape-standard} best fit the data.
$$a=988.1532 \quad \mbox{and} \quad b= 0.0219,$$
and the error is given by
$$E_{classical}\approx16.7403.$$
When we consider the fractional approach, with fractional order $\alpha\in(0,2)$ for Function  \eqref{tape-farc-1}, we obtain the values
$$\alpha=0.9917, \quad p=10.9666  \quad \mbox{and} \quad b=  0.0199$$
with error
$$E_{fractional}\approx16.2894.$$
For this last case, we have
$$\frac{16.7403-16.2894}{16.7403}\approx 0.03 .$$

The graph of Figure \ref{fig:FigVideotape} illustrates Video Tape model with experimental data, classical model and fractional model.
\\
\begin{figure}[!htbp]
\centering
\hspace*{-15pt}\includegraphics[width=1\linewidth]{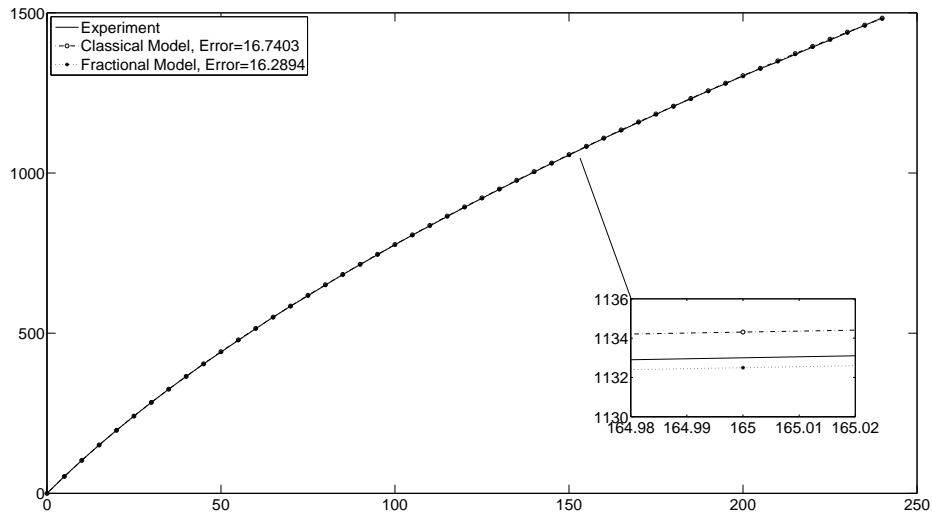}
\caption{The Video Tape problem.}
\label{fig:FigVideotape}
\end{figure}

\section{Conclusions}

Fractional differential equations can describe better certain real world phenomena. This is understandable since systems are not usually perfect, and can be perturbed (like friction, manipulation, external forces, etc) and because of it integer-order derivatives may not be adequate to understand the trajectories of the state variables. By considering fractional derivatives, we have an infinite choices of derivative orders that we can consider, and with it determine what is the fractional differential equation that better describes the dynamics of the model. We have seen this in our experimental data, where non-integer derivatives allow the solution curve to model more efficiently the problems.

\section*{Acknowledgements}

This work was supported by Portuguese funds through the CIDMA - Center for Research and Development in Mathematics and Applications,
and the Portuguese Foundation for Science and Technology (FCT-Funda\c{c}\~ao para a Ci\^encia e a Tecnologia), within project UID/MAT/04106/2013; third author by the ALGORITMI R\&D Center and project PEst-UID/CEC/00319/2013. The authors are very grateful to two referees
for many constructive comments and remarks.

\end{document}